\documentclass[a4paper,10pt]{article}

\usepackage{authblk}
\usepackage{amsmath}
\usepackage{graphicx}
\usepackage{amssymb}
\usepackage{amsthm}
\usepackage{a4wide}
\usepackage[mathscr]{eucal}
\usepackage{mathtools}
\usepackage{bm}
\usepackage{bbm}
\usepackage{enumitem}
\usepackage{hyperref}
\usepackage{booktabs}
\usepackage[nocompress]{cite}
\usepackage{subcaption}

\newcommand{\sss}{\scriptscriptstyle}

\newcommand{\pprob}{\mathbb{P}}
\newcommand{\Prob}[1]{\pprob\left(#1\right)}
\newcommand{\Probn}[1]{\pprob_n\left(#1\right)}

\newcommand{\expec}{\mathbb{E}}
\newcommand{\Exp}[1]{\expec\left[#1\right]}
\newcommand{\Expn}[1]{\expec_n\left[#1\right]}

\newcommand{\Varn}[1]{\textup{Var}_n\left(#1\right)}

\newcommand{\plim}{\ensuremath{\stackrel{\pprob}{\longrightarrow}}}
\newcommand{\dlim}{\ensuremath{\stackrel{d}{\longrightarrow}}}

\newcommand{\ind}[1]{\mathbbm{1}_{\left\{#1\right\}}}

\newcommand{\bigO}[1]{O\left(#1\right)}
\newcommand{\bigOp}[1]{O_{\sss\pprob}\left(#1\right)}

\newcommand\abs[1]{\left|#1\right|}
\newcommand{\me}{\textup{e}}
\newcommand{\dd}{{\rm d}}

\newcommand{\Der}{D^{\sss\mathrm{(er)}}}

\newcommand{\op}{o_{\sss\pprob}}
\newcommand{\Mn}{M^{\sss(n)}}

\newtheorem{theorem}{Theorem}[section]
\newtheorem{definition}{Definition}[section]
\newtheorem{lemma}[theorem]{Lemma}
\newtheorem{proposition}[theorem]{Proposition}

\newtheorem{remark}[definition]{Remark}

\let\OLDthebibliography\thebibliography
\renewcommand\thebibliography[1]{
	\OLDthebibliography{#1}
	\setlength{\parskip}{0pt}
	\setlength{\itemsep}{0pt plus 0.3ex}
}

	\author{Clara Stegehuis}
	
\begin{document}
	\title{Degree correlations in scale-free null models}	
		\affil{Eindhoven University of Technology}

\graphicspath{{Figures/}}

\maketitle
\begin{abstract}
	We study the average nearest neighbor degree $a(k)$ of vertices with degree $k$. In many real-world networks with power-law degree distribution $a(k)$ falls off in $k$, a property ascribed to the constraint that any two vertices are connected by at most one edge. We show that $a(k)$ indeed decays in $k$ in three simple random graph null models with power-law degrees: the erased configuration model, the rank-1 inhomogeneous random graph and the hyperbolic random graph. We find that in the large-network limit for all three null models $a(k)$ starts to decay beyond $n^{(\tau-2)/(\tau-1)}$ and then settles on a power law $a(k)\sim k^{\tau-3}$, with $\tau$ the degree exponent.	\end{abstract}

\section{Introduction}
Complex networks are often studied through mathematical analysis of null models that can match the network degree distribution. For scale-free networks, this degree distribution follows a power law. In many real-world networks, like the Internet, social networks and biological networks, the power-law exponent $\tau$ is found to be between 2 and 3~\cite{albert1999,faloutsos1999,jeong2000,vazquez2002}. In such scale-free networks, high-degree vertices called hubs are likely present, and give rise to scale-free properties such as ultra-small distances and ultra-fast information spreading. The hubs also crucially influence local properties such as clustering~\cite{stegehuis2017,hofstad2017b} and the occurrence of subgraphs~\cite{ostilli2014}. Clustering can be measured in terms of the probability $c(k)$ that a degree-$k$ vertex creates triangles. Both empirically~\cite{pastor2001b,maslov2004} and theoretically~\cite{colomer2012,stegehuis2017} it was shown that $c(k)$ falls off with $k$, and hence that hubs are less likely to take part in triangles. 

Whereas triangles and even larger subgraphs require to study the correlation between at least three vertices, we study in this paper the degree correlation between pairs of two vertices in terms of $a(k)$, the average degree of a neighbor of a vertex of degree $k$.  
According to several studies~\cite{catanzaro2005,barabasi2016}, this degree-degree correlation is an essential local network property, because it also falls off with $k$ and can largely explain the fall-off of $c(k)$~\cite{catanzaro2005,boguna2003,serrano2006}.
We provide support for this statement, by identifying an explicit relation between
$a(k)$ and $c(k)$ for large $k$. But the main goal of this paper is to explain the full spectrum $k\mapsto a(k)$ for all $k$, and to provide theoretical underpinning for the widely observed $a(k)$ fall-off. 

There exist a vast array of papers, empirical, non-rigorous and rigorous, on $a(k)$~\cite{yao2017,catanzaro2005,pastor2001b,park2003,barabasi2016,boguna2003,mayo2015,barrat2004,boguna2004,vazquez2003}. The function $k\mapsto a(k)$ describes the correlation between the degrees on the two sides of an edge, and  classifies the network into one of the following three categories~\cite{newman2002assortative}.
When $a(k)$ increases with $k$, the network is said to be \emph{assortative}: vertices with high degrees mostly connect to other vertices with high degrees. When $a(k)$ decreases in $k$, the network is said to be \emph{disassortative}. Then high-degree vertices typically connect to low-degree vertices. When $a(k)$ is independent of $k$, the network is said to be \emph{uncorrelated}. In this case, the degrees on the two different sides of an edge can be viewed as fully independent, a desirable property when studying the mathematical properties of networks. But the fact is that the majority of real-world networks with power-law degrees and unbounded degree fluctuations ($\tau\in(2,3)$) show a clear decay of $a(k)$ as $k$ grows large~\cite{pastor2001b, maslov2004}. Figure~\ref{fig:akyoutube} illustrates this for the Youtube friendship network~\cite{snap}. Hence, scale-free networks are inherently disassortative, and hubs are predominantly connected to small-degree nodes. In complex network theory, such a well established empirical fact then asks for a theoretical explanation. Typically, this explanation comes in the form of a null model that only matches the degree distribution and has the empirical observation as a property, in this case disassortivity, or more specifically, the essential features of the curve $k\mapsto a(k)$.

\begin{figure}[tb]	
	\centering
	\includegraphics[width=0.45\textwidth]{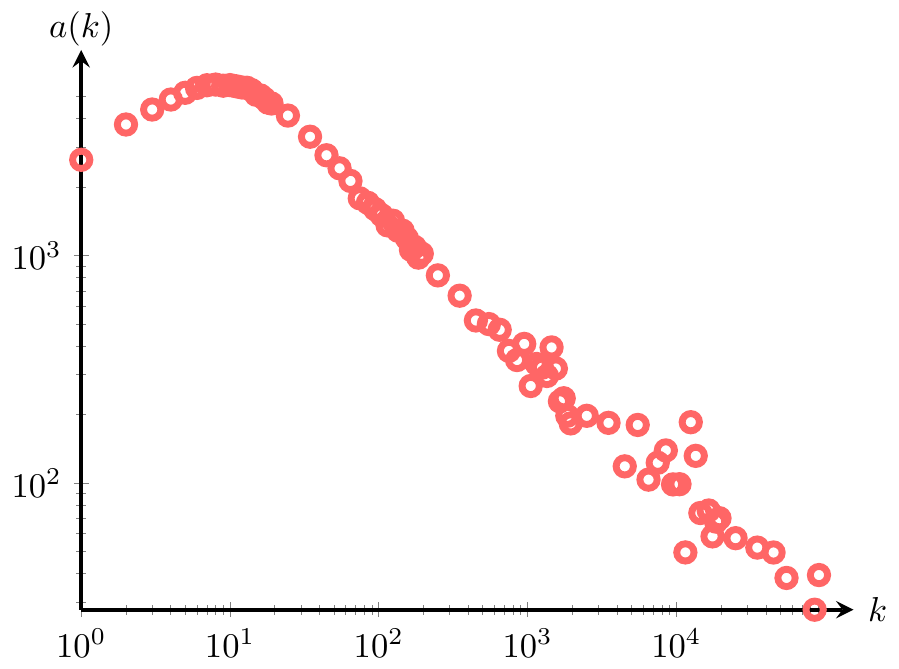}
	\caption{$a(k)$ for the Youtube friendship network~\cite{snap}}
	\label{fig:akyoutube}
\end{figure}
	
The popular configuration model~\cite{bollobas1980} generates random networks with any prescribed degree distribution, but only results in uncorrelated networks when including self-loops and multi-edges. Hence, the configuration model can never explain the $a(k)$ fall-off. We therefore resort to different null models that, contrary to the configuration model, generate random networks without self-loops and multi-edges. The resulting {\it simple} random networks are therefore prone to the structural correlations that come with the presence of hubs. We study $a(k)$ for three widely used null models: the erased configuration model, the rank-1 inhomogeneous random graph (also called hidden variable model) and the hyperbolic random graph. We show that these models display universal $a(k)$-behavior: For $k$ sufficiently small, $a(k)$ is independent of $k$. Thus, in simple scale-free networks, small-degree vertices have similar neighbors. We then identify the value of $k$ as of which $a(k)$ starts decaying. An intuitive explanation for the $a(k)$ fall-off is that in simple networks, high-degree vertices have so many neighbors that they must reach out to lower-degree vertices, because networks typically only contain a small amount of high-degree vertices. This causes the average degree of a neighbor of a high-degree vertex to be smaller. Thus, single-edge constraints may cause the decay of $a(k)$. 

\section{Main results}
We first define the average nearest neighbor degree $a(k,G)$ of a graph $G$ in more detail. Let $(D_i)_{i\in [n]}$ be the degree sequence of the graph, where $[n]=1,\ldots,n$. Furthermore, let $N_k$ denote the total number of degree $k$ vertices in the graph, and $\mathcal{N}_i$ denote the neighborhood of vertex $i$. The average nearest neighbor degree of graph $G$ is then defined as
	\begin{equation}\label{eq:ak}
		a(k,G)=\frac{1}{k N_k}\sum_{i:D_i=k}\sum_{j\in\mathcal{N}_i}D_j.
	\end{equation}
	Note that it is possible that no vertex of degree $k$ exists in the graph. We therefore analyze
	\begin{equation}
		a_\varepsilon(k,G)=\frac{1}{k|M_\varepsilon(k)|}\sum_{i\in M_\varepsilon(k)}\sum_{j\in\mathcal{N}_i}D_j,
	\end{equation}
	where $M_\varepsilon(k)=\{i\in[n]:D_i\in[k(1-\varepsilon),k(1+\varepsilon)]\}$. We will show that in the models we analyze, $M_{\varepsilon}(k)$ is non-empty with high probability, so that $a_{\varepsilon}(k)$ is well defined. Note that $a(k,G)=a_0(k,G)$. 
	We now analyze $a_\varepsilon(k,G)$, first for the erased configuration model in Subsection~\ref{sub1} and then for the rank-1 inhomogeneous random graph and the hyperbolic random graph in Subsection~\ref{sub2}. 
	
		\subsection{The erased configuration model}\label{sub1}
	Given a positive integer $n$ and a degree sequence $(D_1,D_2,\ldots, D_n)$ such that the sum of the degrees is even, the configuration model is a (multi)graph where vertex $i$ has degree $D_i$~\cite{bollobas1980}. We start with $D_j$ free half-edges adjacent to vertex $j$, for $j=1, \ldots, n$. The configuration model is then constructed by pairing free half-edges uniformly at random into edges, until no free half-edges remain. Conditionally on obtaining a simple graph, the resulting graph is a {uniform} graph with the prescribed degrees. This is why the configuration model is often used as a {null model} for real-world networks with given degrees. 
	When the degree distribution has an infinite second moment however, the probability of obtaining a simple graph tends to zero as $n$ grows large (see e.g., \cite[Chapter 7]{hofstad2009}). In this setting the configuration model cannot be used as a null model for simple real-world networks anymore. 
	The erased configuration model is the model where all multiple edges are merged and all self-loops are removed~\cite{britton2006}. Where the configuration model has hard constraints on the degrees but does not create a simple graph, the erased configuration model generates a simple graph while putting soft constraints on the degrees. In particular, we take the original degree sequence to be an i.i.d.\ sample from the distribution
	\begin{equation}
		\label{D-tail}
		\pprob(D=k)=c k^{-\tau}(1+o(1)),\ \quad \text{when }k\to\infty,
	\end{equation}
	 where $\tau\in(2,3)$ so that $\expec[D^2]=\infty$. We denote $\Exp{D}=\mu$. When this sample constructs a degree sequence such that the sum of the degrees is odd, we add an extra half-edge to the last vertex. This does not affect our computations. We denote the actual degree sequence of the graph after merging the multiple edges and self-loops by $(\Der)_{i\in[n]}$, and we call these the erased degrees.
	 
	 \paragraph{Stable random variables.}
	 The limit theorem of $a(k,G_n)$ for the erased configuration model contains stable random variables. A random variable follows a stable distribution if for any positive numbers $a_1$ and $a_2$, there exists a real number $b_1=b_1(a_1,a_2)$ and a positive number $b_2=b_2(a_1,a_2)$ such that 
	 \begin{equation}
		 a_1X_1+a_2X_2\overset{d}{=}b_1+b_2X,
	 \end{equation}
	 where $X_1$ and $X_2$ are independent copies of $X$. Stable random variables can be parametrized by four parameters, and are usually denoted by $\mathcal{S}_\alpha(\sigma,\beta,\mu)$ (see for example~\cite[Chapter 4]{whitt2006}). Throughout this paper, we will only use stable distributions with $\sigma=1,\beta=1,\mu=0$ and we denote $\mathcal{S}_\alpha=\mathcal{S}_\alpha(1,1,0)$ to ease notation.

	We now state the main result for the erased configuration model: 
\begin{theorem}[$a(k,G)$ in the erased configuration model]\label{thm:ak}
		Let $(G_n)_{n\geq 1}$ be a sequence of erased configuration models on $n$ vertices, where the degrees are an i.i.d.\ sample from~\eqref{D-tail}. Take $\varepsilon_n$  such that $\lim_{n\to\infty}\varepsilon_n=0$ and $\lim_{n\to\infty}n^{-1/(\tau-1)} k \varepsilon_n=\infty$ and let $\Gamma$ denote the Gamma function. 
	\begin{enumerate}[label=(\roman*)]
		\item 
  For $k \ll n^{(\tau-2)/(\tau-1)}$,
	\begin{equation}\label{eq:aksmall}
		\frac{a_{\varepsilon_n}(k,G_n)}{n^{(3-\tau)/(\tau-1)}}\dlim \frac{1}{\mu}\left(\frac{2c \Gamma(\tfrac{5}{2}-\tfrac{1}{2}\tau )}{ (\tau-1)(3-\tau)}\cos\left(\frac{\pi(\tau-1)}{4}\right)\right)^{2/(\tau-1)} \mathcal{S}_{(\tau-1)/2},
	\end{equation}
where $\mathcal{S}_{(\tau-1)/2}$ is a stable random variable.
\item
	For $ n^{(\tau-2)/(\tau-1)}\ll k \ll n^{1/(\tau-1)}$, 
	\begin{equation}\label{cc1}
		\frac{a_{\varepsilon_n}(k,G_n)}{n^{3-\tau}k^{\tau-3}}\plim -c \mu^{2-\tau}\Gamma(2-\tau).
	\end{equation}
\end{enumerate}
\end{theorem}
\begin{remark}\label{rem:joint}
	The convergence in~\eqref{eq:aksmall} also holds jointly in $k$ and $n$, so that for $m\geq 1$ and $1\leq k_1<k_2<\dots <k_m \ll n^{(\tau-2)/(\tau-1)}$,
	\begin{equation}\label{eq:akjoint}
	\frac{(a_{\varepsilon_n}(k_1,G_n),\dots,a_{\varepsilon_n}(k_m,G_n))}{n^{(3-\tau)/(\tau-1)}}\dlim  \frac{1}{\mu}\left(\frac{2c \Gamma(\tfrac{5}{2}-\tfrac{1}{2}\tau )}{ (\tau-1)(3-\tau)}\cos\left(\frac{\pi(\tau-1)}{4}\right)\right)^{2/(\tau-1)} \mathcal{S}_{(\tau-1)/2}\boldsymbol{1},
	\end{equation}
	where $\boldsymbol{1}\in\mathbb{R}^m$ is a vector with $m$ entries equal to 1.
\end{remark}

\begin{figure}[tb]
	\centering
	\includegraphics[width=0.4\textwidth]{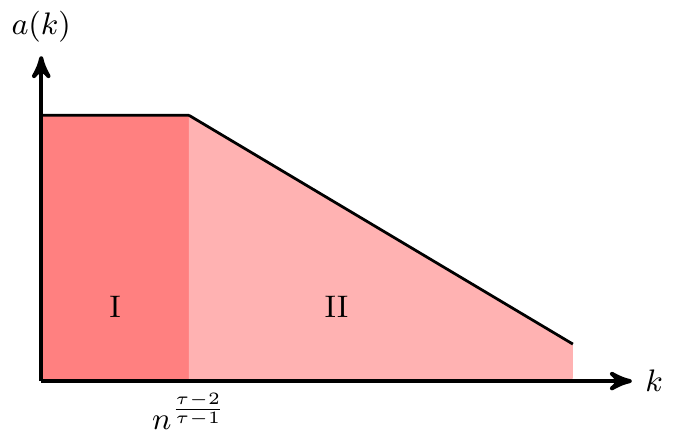}
	\caption{Illustration of the behavior of $a(k,G_n)$ in the erased configuration model}
	\label{fig:acurve}
\end{figure}
Figure~\ref{fig:acurve} illustrates the behavior of $a(k,G_n)$. 
First, $a(k,G_n)$ stays flat and does not depend on $k$. After that, $a(k,G_n)$ starts decreasing in $k$, which shows that the erased configuration model indeed is a disassortative random graph. Theorem~\ref{thm:ak} shows that $n^{(\tau-2)/(\tau-1)}$ serves as a threshold. Thus, the negative degree-degree correlations due to the single-edge constraint only affect vertices of degrees at least $n^{(\tau-2)/(\tau-1)}$. 
This can be understood as follows. In the erased configuration model the maximum contribution to $a(k,G)$ (see Propositions~\ref{prop:minor} and~\ref{prop:akmajor}) comes from vertices with degrees proportional to $n/k$. The maximal degree in an observation of $n$ i.i.d.\ power-law distributed samples is proportional to $n^{1/(\tau-1)}$ w.h.p. Therefore, if $k\ll n^{(\tau-2)/(\tau-1)}$, such vertices with degree proportional to $n/k$ do not exist w.h.p. This explains the two regimes.

 For $k$ small, $a(k,G_n)$ converges to a stable random variable, as was also shown in~\cite{yao2017} for $k$ fixed. Thus, for $k$ small, different instances of the erased configuration model show wild fluctuations. The joint convergence in $k$ of Remark~\ref{rem:joint} shows that $a(k,G_n)$ still forms a flat curve in $k$ for one realization of an erased configuration model when $k$ is small. In contrast, $a(k,G_n)$ converges to a constant for large $k$-values, so that different realizations of erased configuration models will result in similar $a(k,G_n)$-values.

\subsection{Sketch of the proof}
We now give a heuristic proof of Theorem~\ref{thm:ak}. Conditionally on the degrees, the probability that vertices with degrees $D_i$ and $D_j$ are connected in the erased configuration model can be approximated by~\cite{hofstad2005}
\begin{equation}\label{appp}
1-\me^{-D_iD_j/\mu n}.
\end{equation}
Let $v\in M_{\varepsilon_n}(k)$, and let $X_{ik}$ denote the indicator that vertex $i$ is connected to $v$. The expected degree of a neighbor of $v$ can then be approximated by
\begin{equation}\label{eq:akheuristic}
a_{\varepsilon_n}(k,G_n)\approx k^{-1}\sum_{i\in[n]}{D_i\Prob{X_{iv}=1}} \approx  k^{-1}\sum_{i\in [n]}D_i(1-\me^{-D_ik/(\mu n)}).
\end{equation}
The maximum degree in an i.i.d.\ sample from~\eqref{D-tail} scales as $n^{1/(\tau-1)}$ w.h.p.. Thus, as long as $k\ll n^{(\tau-2)/(\tau-1)}$, we can Taylor expand the exponential so that 
\begin{equation}
a_{\varepsilon_n} (k,G_n)\approx  \frac{1}{\mu n}\sum_{i\in [n]}D_i^2.
\end{equation}
Because $(D_i)_{i\in[n]}$ are samples from a power-law distribution with infinite second moment, the Stable Law Central Limit Theorem gives Theorem~\ref{thm:ak}(i). 

When $k\gg n^{(\tau-2)/(\tau-1)}$, we approximate the sum in~\eqref{eq:akheuristic} by the integral 
\begin{equation*}
a_{\varepsilon_n} (k,G_n)\approx cnk^{-1}\int_{1}^{\infty}x^{1-\tau}(1-\me^{-xk/(\mu n)})\dd x = c\mu^{2-\tau}\left(\frac{n}{k}\right)^{3-\tau}\int_{k/(\mu n)}^{\infty}y^{1-\tau}(1-\me^{-y})\dd y,
\end{equation*}
using the degree distribution~\eqref{D-tail} and the change of variables $y=xk/(\mu n)$.
When $ k \ll n$, we can approximate this by
\begin{equation}
	 a_{\varepsilon_n} (k,G_n)\approx c\mu^{2-\tau}\left(\frac{n}{k}\right)^{3-\tau}\int_{0}^{\infty}y^{1-\tau}(1-\me^{-y})\dd y=-c\mu^{2-\tau}\left(\frac{n}{k}\right)^{3-\tau}\Gamma(2-\tau).
\end{equation}
The proof of Theorem~\ref{thm:ak}(ii) then consists of showing that the above approximations are indeed valid. 
We prove Theorem~\ref{thm:ak} in detail in Sections~\ref{sec:ksmall} and~\ref{sec:klarge}.

\subsection{Two more null models}\label{sub2}

We now turn to the rank-1 inhomogeneous random graph (or hidden variable model). This model constructs simple graphs with soft constraints on the degree sequence~\cite{chung2002,boguna2003}. The graph consists of $n$ vertices with weights $(h_i)_{i\in[n]}$. These weights are an i.i.d.\ sample from the power-law distribution~\eqref{D-tail}. We denote the average value of the weights by $\mu$. Then, every pair of vertices with weights $(h,h')$ is connected with probability $p(h,h')$. In this paper, we take
	\begin{equation}\label{eq:phh}
		p(h,h')=\min\left({h h'}/(\mu n),1\right),
	\end{equation}
	which is the Chung-Lu version of the rank-1 inhomogeneous random graph~\cite{chung2002}. This connection probability ensures that the degree of a vertex with weight $h$ will be close to $h$~\cite{boguna2003}. We show the following result:
\begin{theorem}[$a(k,G_n)$ in the rank-1 inhomogeneous random graph]\label{thm:akhvm}
	Let $(G_n)_{n\geq 1}$ be a sequence of rank-1 inhomogeneous random graphs on $n$ vertices, where the weights are an i.i.d.\ sample from~\eqref{D-tail}. Take $\varepsilon_n$  such that $\lim_{n\to\infty}\varepsilon_n=0$ and $\lim_{n\to\infty}n^{-1/(\tau-1)} k \varepsilon_n=\infty$ and let $\Gamma$ denote the Gamma function. 
	\begin{enumerate}[label=(\roman*)]
		\item 
		For $1\ll k\ll n^{(\tau-2)/(\tau-1)}$, 
		\begin{equation}
		\frac{a_{\varepsilon_n}(k,G_n)}{n^{(3-\tau)/(\tau-1)}}\dlim \frac{1}{\mu}\left(\frac{2c \Gamma(\tfrac{5}{2}-\tfrac{1}{2}\tau )}{ (\tau-1)(3-\tau)}\cos\left(\frac{\pi(\tau-1)}{4}\right)\right)^{2/(\tau-1)} \mathcal{S}_{(\tau-1)/2},
		\end{equation}
		where $\mathcal{S}_{(\tau-1)/2}$ is a stable random variable.
		\item
		For $n^{(\tau-2)/(\tau-1)} \ll k\ll n^{1/(\tau-1)}$, 
		\begin{equation}\label{cc2}
		\frac{a_{\varepsilon_n}(k,G_n)}{n^{3-\tau}k^{\tau-3}}\plim \frac{c \mu^{2-\tau}}{(3-\tau)(\tau-2)}.
		\end{equation}
	\end{enumerate}
\end{theorem}
Theorem~\ref{thm:akhvm} is almost identical to Theorem~\ref{thm:ak}. The proof of Theorem~\ref{thm:akhvm} exploits the deep connection between both models, and essentially carries over the results for the erased configuration model to the rank-1 inhomogeneous random graph. The similarity can be understood by noticing that in the erased configuration model the probability that vertices $i$ and $j$ with degrees $D_i$ and $D_j$ are connected can be approximated by~\eqref{appp} which is close to $\min(1,\frac{D_iD_j}{\mu n})$, the connection probability in the rank-1 inhomogeneous random graph. Similar arguments that lead to~\eqref{eq:akheuristic} show that $a_{\varepsilon_n}(k,G_n)$ can be approximated by
\begin{equation}
a_{\varepsilon_n}(k,G_n)\approx k^{-1}\sum_{i\in[n]}h_i\min(h_ik/\mu n,1)\dd x.
\end{equation}
This sum behaves very similarly to the sum in~\eqref{eq:akheuristic}, so that the only difference between Theorem~\ref{thm:ak} and~\ref{thm:akhvm} is the limiting constants in \eqref{cc1} and \eqref{cc2}. The main difference between both models is that in the rank-1 inhomogeneous random graph the presence of all edges is independent as soon as the weights are sampled. This is not true in the erased configuration model, because we know that a vertex with sampled degree $D_i$ cannot have more than $D_i$ neighbors, creating dependence between the presence of edges incident to vertex $i$. We show that these correlations between the presence of different edges in the erased configuration model are small enough for $a(k,G_n)$ to behave similarly in the erased configuration model and the rank-1 inhomogeneous random graph.

The third null model we consider is the hyperbolic random graph. This model was introduced in \cite{krioukov2010} and samples $n$ vertices on a disk of radius $R=2\log(n/\nu)$, where the density of the radial coordinate $r$ a vertex $p=(r,\phi)$ is
	\begin{equation}
		\rho(r)=\alpha\frac{\sinh(\alpha r)}{\cosh(\alpha R)-1}
	\end{equation}
	with $\alpha=(\tau-1)/2$. The angle of $p$ is sampled uniformly from $[0,2\pi]$. Then, two vertices are connected if their hyperbolic distance is at most $R$. The hyperbolic distance of points $u=(r_u,\phi_u)$ and $v=(r_v,\phi_v)$ is defined by
	\begin{equation}\label{eq:dhyp}
		\cosh(\dd (u,v))=\cosh(r_u)\cosh(r_v)-\sinh(r_u)\sinh(r_v)\cos(\theta_{uv}),
	\end{equation}
	where $\theta_{uv}$ denotes the relative angle between $\phi_u$ and $\phi_v$.
	This creates a simple random graph with power-law degrees with exponent $\tau$~\cite{krioukov2010}. The parameter $\nu$ fixes the average degree of the graph. 

	The hyperbolic random graph creates simple sparse random graphs with power-law degrees, but in contrast to the erased configuration model and the rank-1 inhomogeneous random graph, can at the same time
	create many triangles due to its geometric nature~\cite{krioukov2010,candellero2014}. 
	In both the rank-1 inhomogeneous random graph and the erased configuration model, the connection probabilities of different pairs of vertices are (almost) independent. In the hyperbolic random graph, this is not true. When $u$ is connected to $v$ and $u$ is connected to $w$, then $v$ and $w$ should also be close to one another by the triangle inequality. However, if we define 
	\begin{equation}\label{eq:tu}
	t(u)=\me^{(R-r_u)/2}
	\end{equation}
	then we show that we can approximate the probability that two randomly chosen vertices $u$ and $v$ are connected as
	\begin{equation}\label{eq:puvhyper}
	\Prob{X_{uv}=1\mid t(u),t(v)}=\min\Big(\frac{1}{\pi}\cos^{-1}(1-2(\nu t(u)t(v)/n)^2),1\Big),
	\end{equation}
	which behaves similarly as the connection probability in the rank-1 inhomogeneous random graph. Furthermore, by~\cite[Lemma 1.3]{bode2015}, the density of $2\ln(t(u))$ can be written as
	\begin{equation}
		f_{2\ln(t(u))}(x)=\tfrac{\tau-1}{2}\me^{-(\tau-1)x/2}(1+o(1)),
	\end{equation}
	where the $o(1)$ term is with respect to the network size $n$. Therefore, 
	\begin{equation}\label{eq:tudistr}
	\Prob{t(u)>x}=\Prob{2\ln(t(u))>2\ln(x)}=x^{-\tau+1}(1+o(1)),
	\end{equation}
	so that on a high level the hyperbolic random graph can be interpreted as a rank-1 inhomogeneous random graph with $(t(u))_{u\in[n]}$ as weights (see~\cite[Section 1.1.1]{bode2015} for a more elaborate discussion).

	The next theorem shows that indeed the behavior of $a(k,G_n)$ in the hyperbolic random graph is similar as in the rank-1 inhomogeneous random graph:
\begin{theorem}[$a(k,G_n)$ in the hyperbolic random graph]\label{thm:akhrg}
	Let $(G_n)_{n\geq 1}$ be a sequence of hyperbolic random graphs on $n$ vertices with power-law degrees with exponent $\tau$ and parameter $\nu$. Take $\varepsilon_n$  such that $\lim_{n\to\infty}\varepsilon_n=0$ and $\lim_{n\to\infty}n^{-1/(\tau-1)} k \varepsilon_n=\infty$ and let $\Gamma$ denote the Gamma function. 
	\begin{enumerate}[label=(\roman*)]
		\item 
		For $1\ll k\ll n^{(\tau-2)/(\tau-1)}$,
		\begin{equation}
		\frac{a_{\varepsilon_n}(k,G_n)}{n^{(3-\tau)/(\tau-1)}}\dlim\frac{2\nu}{\pi}\left(\frac{2}{3-\tau}\Gamma(\tfrac{5}{2}-\tfrac{1}{2}\tau )\cos\left(\frac{\pi(\tau-1)}{4}\right)\right)^{2/(\tau-1)}\mathcal{S}_{(\tau-1)/2},
		\end{equation}
		where $\mathcal{S}_{(\tau-1)/2}$ is a stable random variable.
		\item
		For $ n^{(\tau-2)/(\tau-1)}\ll k \ll n^{1/(\tau-1)}$, 
		\begin{equation}
		\frac{a_{\varepsilon_n}(k,G_n)}{n^{3-\tau}k^{\tau-3}}\plim\frac{\nu (\tau-1)^2}{(\tau-2)\pi}\left(\frac{\pi}{2\nu}\right)^{2-\tau}\int_{0}^{\infty}x^{1-\tau}\min\left(\frac{1}{\pi}\cos^{-1}(1-2x^2),1\right)\dd x .
		\end{equation}
	\end{enumerate}
\end{theorem}

\subsection{Discussion}\label{sub3}

\paragraph{Universality.} 
The behavior of $a(k)$ is universal across the three null models we consider. The erased configuration model and the rank-1 inhomogeneous random graph are closely related. They are known to behave similarly for example under critical percolation~\cite{bhamidi2014,bhamidi2017}, in terms of distances~\cite{esker2008} when $\tau>3$, and in terms of clustering when $\tau\in(2,3)$~\cite{stegehuis2017}. The hyperbolic random graph typically shows different behavior, for example in terms of clustering~\cite{gugelmann2012,candellero2014}, or connectivity~\cite{bode2015,bode2016}. Still, the behavior of $a(k)$ is similar in the hyperbolic random graph and the other two null models.
In all three null models, the main contribution to $a(k)$ for $k\gg n^{(\tau-2)/(\tau-1)}$ comes from vertices with degrees proportional to $n/k$ (see Propositions~\ref{prop:minor} and~\ref{prop:akmajor}). In the hyperbolic random graph, we can relate this maximum contribution to the geometry of the hyperbolic sphere. A vertex $i$ of degree $k$ has radius $r_i\approx R-2\log(k)$. Similarly, a vertex $j$ of degree $ n/(\nu k)$ has radius $r_j\approx R-2\log( n/(k\nu))=2\log(k)$. Then, $r_j\approx R-r_i$, so that the major contributing vertices have radial coordinate proportional to $R-r_i$.

\paragraph{Expected average nearest neighbor degree.}
In Theorems~\ref{thm:ak}-\ref{thm:akhrg} we show that $a(k,G_n)$ converges in probability to a stable random variable when $k$ is small. Thus, when we generate many samples of random graphs, we will see that for fixed $k$, the distribution of the values of $a(k,G_n)$ across the different samples will look like a stable random variable. We can also study the expected value of $a(k,G_n)$ across the different samples. For the erased configuration model for example, using similar techniques as in the proof of Theorem~\ref{thm:ak}(ii), we can show that (see Section~\ref{sec:akexp})
\begin{equation}\label{eq:expak}
\lim_{n\to\infty}\frac{\Exp{a(k,G_n)}}{(n/k)^{3-\tau}}=-c\mu^{2-\tau}\Gamma(2-\tau).
\end{equation}
 The difference between the scaling of the expected value of $a(k,G_n)$ and the typical behavior of $a(k,G_n)$ in Theorem~\ref{thm:ak}(i) is caused by high-degree vertices. In typical degree sequences, the maximum degree is proportional to $n^{1/(\tau-1)}$. It is unlikely that vertices with higher degrees are present, but if they are, they have a high impact on the average nearest neighbor degree of low degree vertices, causing the difference between the expected average nearest neighbor degree and the typical average nearest neighbor degree. Thus, the expected value of $a(k,G_n)$ is not very informative when $k$ is small, since Theorem~\ref{thm:ak} shows that $a(k,G_n)$ will almost always be smaller than its expected value when $k$ is small. Also note that including $\varepsilon_n$ in~\eqref{eq:expak} is not necessary, since the expected value is not affected by the event that no vertex of degree $k$ is present.

Figure~\ref{fig:ak} illustrates this difference in terms of the mean and median value of $a(k,G_n)$ over many realizations of the erased configuration model, the rank-1 inhomogeneous random graph and the hyperbolic random graph. Here indeed we see that the expected average neighbor degree is scales as a power of $k$ over the entire range of $k$, where the median shows the straight part of the curve from Theorem~\ref{thm:ak}. Thus, it is important to distinguish between mean and median of $a(k,G_n)$ when simulating random graphs.
\begin{figure}[tb]
	\centering
	\begin{subfigure}[t]{0.32\textwidth}
		\centering
		\includegraphics[width=\linewidth]{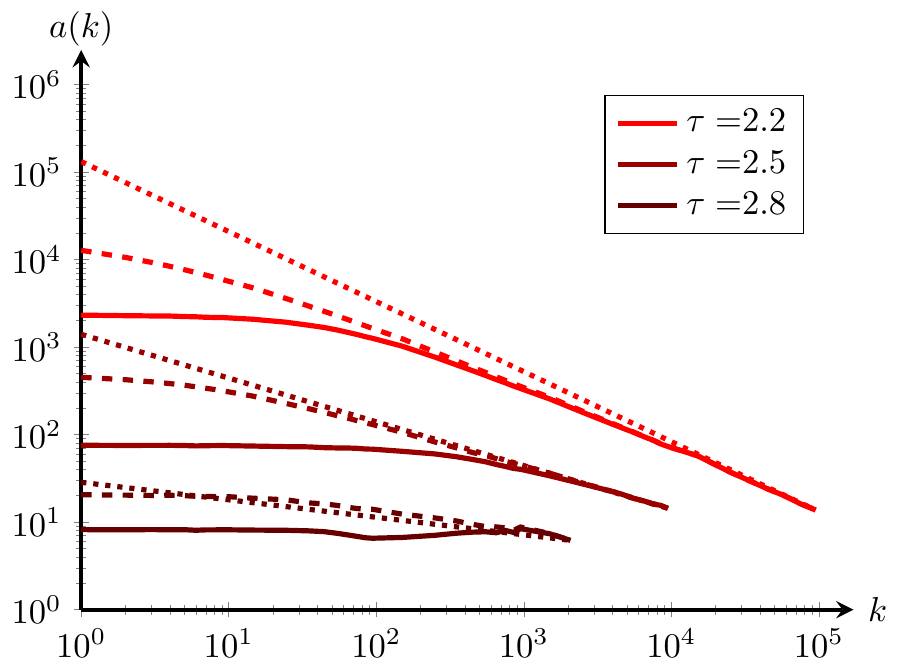}
		\caption{erased configuration model}
		\label{fig:akECM}
	\end{subfigure}
	\begin{subfigure}[t]{0.32\textwidth}
		\centering
	\includegraphics[width=\linewidth]{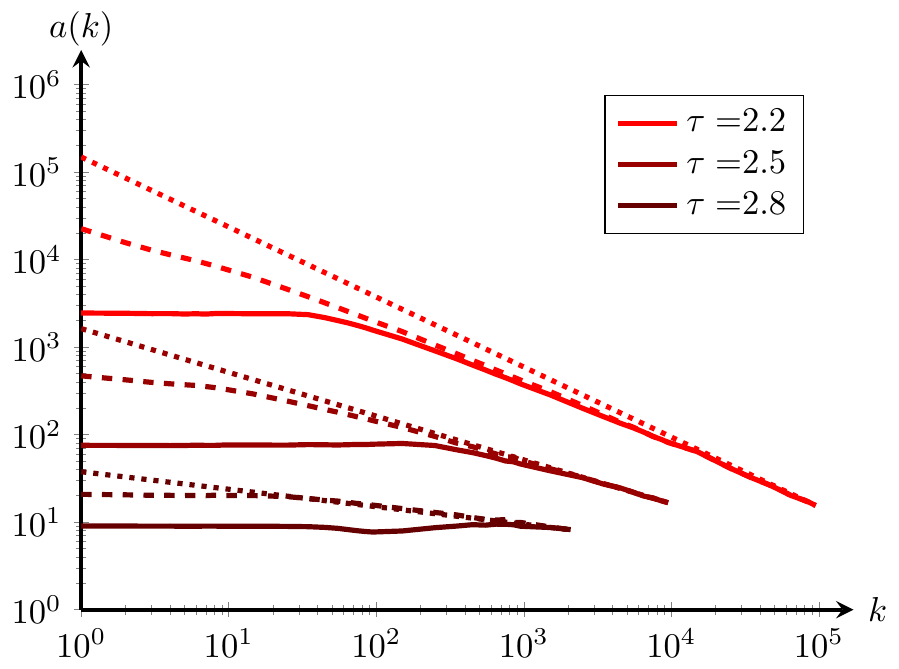}
	\caption{inhomogeneous random graph}
	\label{fig:akhidden}
	\end{subfigure}
\begin{subfigure}[t]{0.32\textwidth}
	\centering
	\includegraphics[width=\linewidth]{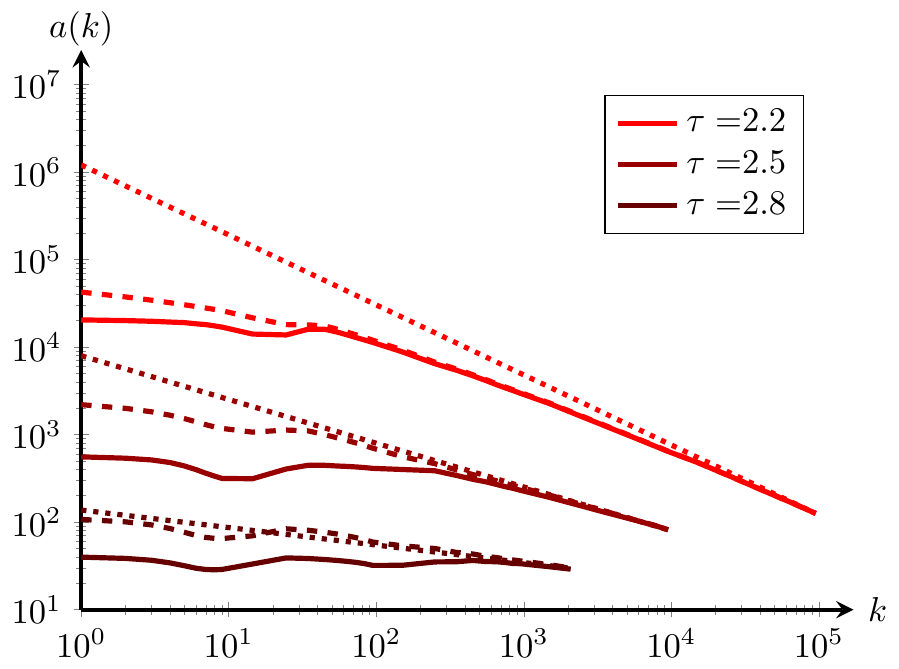}
	\caption{hyperbolic random graph}
	\label{fig:akhyper}
\end{subfigure}
	\caption{$a(k,G_n)$ for different random graph models with $n=10^6$. The solid line is the median of $a(k,G_n)$ over $10^4$ realizations of the random graph, and the dashed line is the average over these realizations. The dotted line is the asymptotic slope $k^{\tau-3}$.}
	\label{fig:ak}
\end{figure}

\paragraph{Vertices of degree $k$}
The definition~\eqref{eq:ak} assumes that a vertex of degree $k$ is present. For large values of $k$, this is a rare event, by~\eqref{D-tail}. Indeed, vertices of degree at most $n^{1/\tau}$ are present with high probability in the erased configuration model, whereas the probability that a vertex of degree $k\gg n^{1/\tau}$ is present tends to zero in the large network limit~\cite{yao2017}. We avoid this problem by averaging $a(k,G_n)$ over a small range of degrees. Another option is to condition on the event that a vertex of degree $k$ is present. Our proofs for $k\ll n^{(\tau-2)/(\tau-1)}$ for the erased configuration model can easily be adjusted to condition on this event. For $k$ larger, we leave the behavior of $a(k,G_n)$ conditionally on a vertex of degree $k$ being present open for further research.

\paragraph{Fixed degrees.}
In the proof of Theorem~\ref{thm:ak} we show that the fluctuations that come with the stable laws for small $k$ are not present when we condition on the degree sequence. Thus, the large fluctuations in $a(k,G_n)$ for small $k$ are only caused by fluctuations of the i.i.d.\ degrees, weights or radii. For a given real-world network, the network degrees are often preserved, and many samples of erased configuration models or inhomogeneous random graphs are created with the observed degree sequence. In this fixed-degree setting, the sample-to-sample fluctuations of $a (k)$ are relatively small.

\paragraph{Relation with local clustering.}
The local clustering coefficient $c(k)$ of vertices of degree $k$ measures the probability that two randomly chosen neighbors af a randomly chosen vertex of degree $k$ are connected. In many real-world networks as well as simple null models, $c(k)$ decreases as a function of $k$~\cite{vazquez2002,boguna2003,ravasz2003,stegehuis2017,hofstad2017b}. The relation between the decay rate of $c(k)$ and the decay rate of $a(k)$ has been investigated for the rank-1 inhomogeneous random graph, where it was shown that $c(k)<a(k)/k$~\cite{serrano2006}. 
Using recent results for $c(k)$ on the erased configuration model and the rank-1 inhomogeneous random graph, we can make the relation between $c(k)$ and $a(k)$ more precise. When $k\gg\sqrt{n}$, $c(k)$ in the erased configuration model satisfies~\cite{hofstad2017b}
\begin{equation}
c(k)=c^2\Gamma(2-\tau)^2\mu^{3-2\tau}n^{5-2\tau}k^{2\tau-6}(1+\op(1)).
\end{equation}
Then, by Theorem~\ref{thm:ak}, when $k\gg \sqrt{n}$,
\begin{equation}\label{eq:ckak}
c(k)=\frac{a^2(k)}{\mu n}(1+\op(1)).
\end{equation}
Intuitively, we can see this relationship in the following way. Pick two neighbors of a vertex with degree $k$. By definition, these vertices have degree $a(k)$ on average. Since $k\gg\sqrt{n}$, by Theorem~\ref{thm:akhvm} $a(k)\ll\sqrt{n}$. Therefore, the probability of two vertices with weight $a(k)$ to be connected is approximately $1-\me^{-a(k)^2/\mu n}\approx a(k)^2/\mu n$. Since the clustering coefficient can be interpreted as the probability that two randomly chosen neighbors are connected, the clustering coefficient should satisfy $c(k)\sim a(k)^2/\mu n$ when $k\gg \sqrt{n}$. In particular, the decay of the clustering coefficient should be twice as fast as the decay of the average neighbor degree. Analytical results on $c(k)$ on the rank-1 inhomogeneous random graph show that~\eqref{eq:ckak} is also the correct relation between clustering and degree correlations in the rank-1 inhomogeneous random graph~\cite{stegehuis2017}. Future research might explore the relation between $c(k)$ and $a(k)$ in other null models, such as the hyperbolic random graph or the preferential attachment model. It would also be interesting to see if the difference between expectation and typical behavior that is present in $a(k)$ also occurs for the local clustering coefficient $c(k)$.
\begin{figure}[tb]
	\centering
	\includegraphics[width=0.3\linewidth]{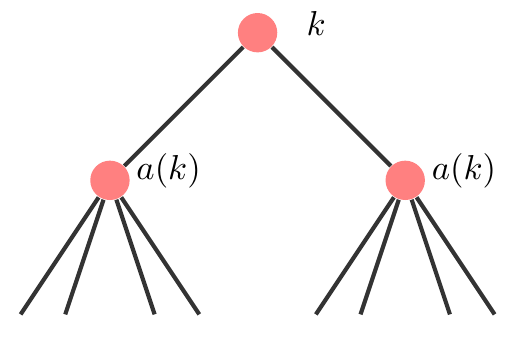}
	\caption{The neighbors of a vertex of degree $k$ have average degree $a(k)$}
	\label{fig:trianga}
\end{figure}

\paragraph{Correlations in the hyperbolic random graph.}
The relation between $a(k)$ and $c(k)$ in the rank-1 inhomogeneous random graph and the erased configuration model is based on the fact that in these two models, the connection probabilities of pairs of vertices $(i,j)$, $(i,k)$ and $(j,k)$ are (almost) independent. In the hyperbolic random graph, the geometry causes a strong dependence between these connection probabilities. If vertices $j$ and $k$ are neighbors of $i$, they are likely to be geometrically close to one another due to the triangle inequality. This makes the probability that $j$ and $k$ are connected larger than in the rank-1 inhomogeneous random graph or the erased configuration model. These correlations do not play a role when computing $a(k)$, since $a(k)$ only involves the connection probability of two different vertices. When computing statistics of the hyperbolic random graph that include three-point correlations, the equivalence between the hyperbolic random graph and the rank-1 inhomogeneous random graph may fail to hold, as in the example of $c(k)$. 

Interestingly, the number of cliques was also shown to be similar in the hyperbolic random graph, the rank-1 inhomogeneous random graph and the erased configuration model~\cite{friedrich2015}, even though cliques clearly involve three-point correlations. 
Cliques in the hyperbolic random graph are typically formed between vertices at radius proportional to $R/2$~\cite{friedrich2015}, so that their degrees are proportional to $\sqrt{n}$~\cite{bode2015}. These vertices form a dense core, which is very similar to what happens in the erased configuration model and the rank-1 inhomogeneous random graph~\cite{janson2006}. In the erased configuration model, many other small subgraphs typically occur between vertices of degrees proportional to $\sqrt{n}$~\cite{hofstad2017d}. It would be interesting to see if the number of these small subgraphs behaves similarly in the hyperbolic random graph.


	\section{Average nearest neighbor degree in the ECM}\label{sec:akecm}
	In this section, we prove Theorem~\ref{thm:ak}. For $k=o(n^{\frac{\tau-2}{\tau-1}})$, we couple the degrees of neighbors of a uniformly chosen vertex of degree $k$ to i.i.d.\ samples of the size-biased degree distribution in Section~\ref{sec:ksmall}. When $k\gg n^{\frac{\tau-2}{\tau-1}}$, this coupling is no longer valid. We then show in Section~\ref{sec:klarge} that a specific range of degrees contributes most to $a_{\varepsilon_n}(k,G_n)$. 
	
	\subsection{Preliminaries}
	 We say that $X_n=\bigOp{b_n}$ for a sequence of random variables $(X_n)_{n\geq 1}$ if $|X_n|/b_n$ is a tight sequence of random variables, and $X_n=\op(b_n)$ if $X_n/b_n\plim 0$. We often want to interchange the sampled degree of a vertex $i$, $D_i$ and its erased degree $\Der_i$. By~\cite[Eq. A(9)]{britton2006}
	\begin{equation}\label{eq:DerD}
		\Der_i=D_i(1+\op(1)),
	\end{equation}
	when $D_i=o(n)$. Let $L_n$ denote the total number of half-edges, so that $L_n=\sum_i D_i$. 
	We define the events
	\begin{equation}\label{eq:jn}
	J_n=\{\abs{L_n-\mu n}\leq n^{2/\tau}\} \quad A_n=\{|M_{\varepsilon_n}|\geq 1\}.
	\end{equation}
	By~\cite[Lemma 2.3]{hofstad2017}, $\Prob{J_n}\to 1$ as $n\to\infty$. By~\cite[Theorem 2.1]{britton2006}
	\begin{equation}
	|M_{\varepsilon_n}(k)|=c n\int_{k(1-\varepsilon_n)}^{k(1+\varepsilon_n)}x^{-\tau}\dd x(1+\op(1)) =\tilde{C}n^{-1}k^{1-\tau}\varepsilon_n(1+\op(1)),
	\end{equation}
	for some $\tilde{C}>0$, so that $\Prob{A_n}\to 1$ for $k\ll n^{1/(\tau-1)}$ by the choice of $\varepsilon_n$ in Theorem~\ref{thm:ak}. 
	
	In the rest of this section, we will often condition on the degree sequence. For some event $\mathcal{E}$, we use the notation $\Probn{\mathcal{E}}=\Prob{\mathcal{E}\mid (D_i)_{i\in[n]}}$, and we define $\expec_n$ and $\text{Var}_n$ similarly.

	\subsection{Small $k$: Coupling to i.i.d.\ random variables}\label{sec:ksmall}
	In this section we investigate the behavior of $a_{\varepsilon_n}(k,G_n)$ when $k=o( n^{(\tau-2)/(\tau-1)})$. We first pick a random vertex $v$ of degree $k$. We couple the degrees of the neighbors of $v$ to i.i.d.\ copies of the size-biased degree distribution $D_n^*$, where
	\begin{equation}
	\Probn{D_n^*=k}=\frac{k}{L_n}\sum_{i\in[n]}\ind{D_i=k}.
	\end{equation}
	 We then use this coupling to compute $a_{\varepsilon_n} (k,G_n)$. 
	
\begin{proof}[Proof of Theorem~\ref{thm:ak}(i)]
	We first condition on the degree sequence $(D_i)_{i\in[n]}$. Let $v$ be a vertex of degree $k$. 
	In the erased configuration model, neighbors of $v$ are constructed by pairing the half-edges of $v$ uniformly to other half-edges. The distribution of the degree of a vertex attached to a uniformly chosen half-edge is given by $D_n^*$. However, the degrees of the neighbors of $v$ are not an i.i.d.\ sample of $D_n^*$ due to the fact that the half-edges should attach to distinct vertices, because the neighbors of $v$ should be distinct vertices. We now show that we can still approximate the degrees of the neighbors of $v$ by an i.i.d.\ sample of $D_n^*$ by using a coupling argument. Denote the degrees of the neighbors of $v$ by $B_1,\dots, B_k$, in the order in which we encounter them. Let $Y_1,\dots, Y_k$ be i.i.d.\ samples of $D_n^*$. These samples can be obtained by sampling uniform half-edges with replacement and setting $Y_i=D_{v'_i}$, where $v'_i$ denotes the vertex incident to the $i$th drawn half-edge.
	We use a similar coupling as in~\cite[Construction 4.2]{bhamidi2012a} to couple the $B_i$ to $Y_i$. Let $(v_i')_{i\in[k]}$ denote vertices attached to $k$ uniformly chosen half-edges (with replacement) and set $V_0=v$.
	 Then for $i\in[k]$ the coupling is defined in the following way:
	\begin{itemize}
		\item
	If $v'_i\notin V_{i-1}$, then $B_i=Y_i$ and $v_i=v'_i$. Set $V_i=V_{i-1}\cup v_i'$. We say that $B_i$ and $Y_i$ are successfully coupled.
	\item 
	If $v'_i \in V_{i-1}$, we redraw a uniformly chosen half-edge from the set of half-edges not incident to $ V_{i-1}$. Let $v_i$ denote the vertex incident to the chosen half-edge. Set $B_i=D_{v_i}$ and $V_i=V_{i-1}\cup v_i$. We then say that $B_i$ and $Y_i$ are miscoupled.
	\end{itemize}
Thus, informally, we select $k$ uniformly chosen half-edges, and look at the vertices they point to. If these vertices are all distinct, we have successfully coupled the neighbors of $v$ to an i.i.d.\ sample of $D_n^*$. If not, we need to redraw some of these half-edges to ensure that all neighbors of $v$ are distinct. We now show that the coupling is successful with high probability. 
	By~\cite[Lemma 4.3]{bhamidi2012a}, the probability of a miscoupling at step $i$ can be bounded as
	\begin{equation}
		\Probn{B_i\neq Y_i\mid \mathcal{F}_{i-1}}\leq {L_n}^{-1}\Big(k+\sum_{s=1}^{i-1}B_s\Big),
	\end{equation}
	where $\mathcal{F}_{i}=\sigma(B_j,Y_j)_{j\in[i]}$ denotes the sigma-algebra containing all information about the $Y_j$ and $B_j$ variables encountered up to step $i$.
	Thus, the expected number of miscouplings up to time $t$, $N_{\text{mis}}(t)$, satisfies
	\begin{equation}
		\Expn{N_{\text{mis}}(t)}\leq \frac{kt}{L_n}+\frac{1}{L_n}\sum_{i=1}^t\sum_{s=1}^{i-1}\Expn{B_s}.
	\end{equation}
	When $B_s$ is successfully coupled, $\Expn{B_s\mid \text{succesfully coupled}}=\Expn{D_n^*}=\sum_i D_i^2/L_n$.
	When $B_s$ is not successfully coupled, it is drawn in a size-biased manner from the vertices that are not chosen yet. Then
	\begin{equation}
		\begin{aligned}[b]
		\Expn{B_s\mid \mathcal{F}_{i-1},\text{ miscoupled}} & =\frac{\sum_{i\notin V_s}D_i^2}{\sum_{i\notin V_s}D_i}\leq \frac{\sum_{i\in[n]}D_i^2}{\sum_{i\in [n]}D_i-\sum_{i\in V_s}D_i}\\
		& =\frac{\sum_{i\in [n]} D_i^2}{\sum_{i\in [n]}D_i}\left(1+\frac{\sum_{i\in V_s}D_i}{\sum_{i\in [n]}D_i-\sum_{i\in V_s}D_i}\right).
		\end{aligned}
	\end{equation}
	Since $D_{\max}=\bigOp{n^{1/(\tau-1)}}$, $\sum_{i\in V_s}D_i=\bigOp{sn^{1/(\tau-1)}}$ for all possible $\mathcal{F}_{i-1}$. 
	For $t$ large, we obtain from~\eqref{D-tail} that
	\begin{equation}\label{eq:Dsq}
	\Prob{D^2>t}=\Prob{D>\sqrt{t}}=\frac{c}{\tau-1}t^{(1-\tau)/2}(1+o(1)).
	\end{equation}
	Using~\eqref{eq:Dsq} we can use the Stable Law Central Limit Theorem (see for example~\cite[Theorem~4.5.2]{whitt2006}) to conclude that
	\begin{equation}\label{eq:stableDsq}
	\frac{\sum_{i\in [n]}D_i^2}{n^{2/(\tau-1)}\left(\frac{2c}{ (\tau-1)(3-\tau)} \Gamma(\tfrac{5}{2}-\tfrac{1}{2}\tau )\cos\left(\frac{\pi(\tau-1)}{4}\right)\right)^{2/(\tau-1)} }\dlim  \mathcal{S}_{(\tau-1)/2},
	\end{equation}
		where $\mathcal{S}_{(\tau-1)/2}$ is a stable random variable.
	Thus, as long as $s=o(n^{(\tau-2)/(\tau-1)})$,
	\begin{equation}\label{eq:EBs}
		\Expn{B_s}=O_{\sss\pprob}\Big(L_n^{-1}{\sum_{i\in[n]}D_i^2}\Big)=\bigOp{n^{(3-\tau)/(\tau-1)}}.
	\end{equation}
	Then, for $k=o(n^{(\tau-2)/(\tau-1)})$
	\begin{equation}
		\Expn{N_{\text{mis}}(k)}=\frac{k^2}{L_n}+\frac{1}{L_n}\bigOp{n^{(3-\tau)/(\tau-1)}}\sum_{i=1}^k (i-1)=\bigOp{k^2n^{2\frac{2-\tau}{\tau-1}}}.
	\end{equation}
	Thus, as long as $k=o(n^{\frac{\tau-2}{\tau-1}})$, 
	\begin{equation}
		\Expn{N_{\text{mis}}(k)}=\op(1).
	\end{equation}
	Then, by the Markov inequality
	\begin{equation}
		\Probn{N_\text{mis}(k)=0}=1-\Probn{N_{\text{mis}}(k)\geq 1}\geq 1-\Expn{N_{\text{mis}}(k)}=1-\op(1).
	\end{equation}
	Thus, when $k=o(n^{(\tau-2/(\tau-1))})$, we can approximate the sum of the degrees of the neighbors of a vertex with degree $k$ by i.i.d.\ samples of the size-biased degree distribution. 
	Because $\Der_i=D_i(1+\op(1))$, conditionally on the degree sequence
	\begin{equation*}
		a_{\varepsilon_n}(k,G_n)=\frac{1}{k |M_{\varepsilon_n}(k)|}\sum_{i: \in M_{\varepsilon_n}(k)}\sum_{j\in\mathcal{N}_i}\Der_j=\frac{1}{k}\expec_n\Big[\sum_{j\in \mathcal{N}_{V_k}}\Der_j\Big]=(1+\op(1))\Expn{D_{\mathcal{N}_{V_k}(U)}},
	\end{equation*}
	where $V_k$ denotes a uniformly chosen vertex in $M_{\varepsilon_n}(k)$, and $\mathcal{N}_{V_k}(U)$ is a uniformly chosen neighbor of vertex $V_k$. Here the second equality holds because the average nearest neighbor degree averages over all neighbors of vertex $j$, and the third equality holds because it also averages over all vertices in $M_{\varepsilon_n}(k)$, together with the fact that $D_{V_k}=k(1+o(1))$ and $\Der_i=D_i(1+\op(1))$ uniformly in $i$. With high probability, we can couple the degrees of neighbors of a uniformly chosen vertex of degree in $[k(1-\varepsilon_n),k(1+\varepsilon_n)]$ to i.i.d\ copies of $D_n^*$. Then, conditionally on the degree sequence,
	\begin{equation}\label{eq:akDsize}
		a_{\varepsilon_n} (k,G_n)=(1+\op(1))\Expn{D_n^*}=(1+\op(1))L_n^{-1}\sum_{i\in [n]}D_i^2.
	\end{equation}
	Note that this expression is independent of $k$. Combining this with~\eqref{eq:stableDsq} results in
	\begin{equation}
		\frac{a_{\varepsilon_n} (k,G_n)}{n^{(3-\tau)/(\tau-1)}}\dlim  \frac{1}{\mu }\left(\frac{2c}{ (\tau-1)(3-\tau)} \Gamma(\tfrac{5}{2}-\tfrac{1}{2}\tau )\cos\left(\frac{\pi(\tau-1)}{4}\right)\right)^{2/(\tau-1)} \mathcal{S}_{(\tau-1)/2}.
	\end{equation}
   The fact that~\eqref{eq:akDsize} is independent of $k$ proves the joint convergence of Remark~\ref{rem:joint}.
	\end{proof}
	
	\subsection{Large $k$}\label{sec:klarge}
	Now we study the value of $a_{\varepsilon_n}(k,G_n)$ when $k\gg n^{(\tau-2)/(\tau-1)}$. We show that there exists a range of degrees $W_n^k(\delta)$ which gives the largest contribution to $a_{\varepsilon_n}(k,G_n)$. For ease of notation, we write $a_{\varepsilon_n}(k)$ for $a_{\varepsilon_n}(k,G_n)$ in this section. We define
	\begin{equation}\label{eq:wkn}
	W_n^k(\delta)=\left\{u:D_u\in [\delta n/k,n/(\delta k)]\right\},
	\end{equation}
	and we write
	\begin{equation}
	\begin{aligned}[b]
	a_{\varepsilon_n}(k) &= \frac{1}{k \abs{M_{\varepsilon_n}(k)}}\sum_{i: \in M_{\varepsilon_n}(k)}\sum_{j\in W_n^k(\delta)}\Der_j+ \frac{1}{k \abs{M_{\varepsilon_n}(k)}}\sum_{i: \in M_{\varepsilon_n}(k)}\sum_{j\notin W_n^k(\delta)}\Der_j\\
	&= : a_{\varepsilon_n}(k,W_n^k(\delta))+a_{\varepsilon_n}(k,\bar{W}_n^k(\delta)),
	\end{aligned}
	\end{equation}
	where $a_{\varepsilon_n}(k,W_n^k(\delta))$ denotes the contribution to $a_{\varepsilon_n} (k)$ from vertices in $W_n^k(\delta)$, and $a_{\varepsilon_n}(k,\bar{W}_n(\varepsilon))$ the contribution from vertices not in $W_n^k(\delta)$. In the rest of this section, we prove the following two propositions, which together show that the largest contribution to $a_{\varepsilon_n}(k)$ indeed comes from vertices in $W_n^k(\delta)$.
	
	\begin{proposition}[Minor contributions]\label{prop:minor}
		There exists $\kappa>0$ such that for $k\gg n^{(\tau-2)/(\tau-1)}$,
		\begin{equation}
		\limsup_{n\to\infty}\frac{\Exp{a_{\varepsilon_n}(k,\bar{W}_n^k(\delta))}}{(n/k)^{3-\tau}} =\bigO{\delta^\kappa}.
		\end{equation}
	\end{proposition}

	\begin{proposition}[Major contributions]\label{prop:akmajor}
		\begin{equation}
		\frac{a_{\varepsilon_n}(k,W_n^k(\delta))}{(n/k)^{3-\tau}}\plim c\mu^{2-\tau}\int_{\delta}^{1/\delta}x^{1-\tau}(1-\me^{-x})\dd x
		\end{equation}
	\end{proposition}
	We now show how these propositions prove part (ii) of Theorem~\ref{thm:ak}.
	
	\begin{proof}[Proof of Theorem~\ref{thm:ak} (ii)]
		By the Markov inequality and Proposition~\ref{prop:minor},
		\begin{equation}
		\frac{a_{\varepsilon_n}(k,\bar{W}_n^k(\delta))}{(n/k)^{3-\tau}}=\bigOp{\delta^\kappa}.
		\end{equation}
		Combining this with Proposition~\ref{prop:akmajor} results in
		\begin{equation}\label{eq:aklimit}
		\frac{a_{\varepsilon_n} (k)}{(n/k)^{3-\tau}}\plim c\mu^{2-\tau}\int_{\delta}^{1/\delta}x^{1-\tau}(1-\me^{-x})\dd x+\bigOp{\delta^\kappa}.
		\end{equation}
		Taking the limit of $\delta\to 0$ then proves the theorem.
	\end{proof}
	The rest of this section is devoted to proving Propositions~\ref{prop:minor} and~\ref{prop:akmajor}.
	
	\subsubsection{Conditional expectation}
	We first compute the expectation of $a_{\varepsilon_n}(k,W_n^k(\delta))$ when we condition on the degree sequence.
	
	\begin{lemma}\label{lem:expk}
		When $k\gg n^{(\tau-2)/(\tau-1)}$, 
		\begin{equation}\label{eq:exk}
		\Expn{a_{\varepsilon_n}(k,W_n^k(\delta))}=\frac{1}{k} \sum_{u\in W_n^k(\delta)}D_u(1-\me^{-D_uk/L_n}) (1+\op(1)).
		\end{equation}
	\end{lemma}
	\begin{proof}
	It suffices to prove the lemma under the event $J_n$ from~\eqref{eq:jn}, since $\Prob{J_n}\to 1$. Thus we may assume that $L_n=\mu n(1+o(1))$. 
	Let $X_{ij}$ denote the indicator $i$ and $j$ are connected. 
	By~\eqref{eq:ak}
	\begin{equation}\label{eq:triangeq}
	\begin{aligned}[b]
	\Expn{a_{\varepsilon_n}(k,W_n^k(\delta))}&=\frac{1}{k\abs{M_{\varepsilon_n}(k)}}\sum_{v:\in M_{\varepsilon_n}(k)}\sum_{u\in W_n^k(\delta)}\Der_u\Probn{X_{uv}=1}
	\end{aligned}
	\end{equation}
	By~\cite[Eq. (4.9)]{hofstad2005}
	\begin{equation}\label{eq:pij}
	\Probn{{X}_{uv}=1}=1-\me^{D_uD_v/L_n}+\bigO{\frac{D_v^2D_u+D_u^2D_v}{L_n^2}}=(1-\me^{D_uD_v/L_n})(1+\op(1)),
	\end{equation}
	where the last step follows because $D_u\in n/k[\delta,1/\delta]$ and  by~\eqref{eq:DerD} $D_v=k(1+\op(1))$ when $v\in M_{\varepsilon_n}(k)$. 
	Further using that $\Der_u=D_u(1+\op(1))$ ,we can write~\eqref{eq:triangeq} as
	\begin{equation}
	\begin{aligned}[b]
	\Expn{a_{\varepsilon_n}(k,W_n^k(\delta))}	&=\frac{(1+\op(1))}{k\abs{M_{\varepsilon_n}(k)}}\sum_{v:\in M_{\varepsilon_n}(k)}\sum_{u\in W_n^k(\delta)}D_u(1-\me^{-D_uk/L_n}\me^{\op(D_u/L_n)})\\
	& =\frac{1}{k}\sum_{u\in W_n^k(\delta)}D_u(1-\me^{-D_uk/L_n}\me^{\op(D_u/L_n)})(1+\op(1))\\
	&=\frac{1}{k}\sum_{u\in W_n^k(\delta)}D_u(1-\me^{-D_uk/L_n})(1+\op(1))
	\end{aligned}
	\end{equation}
	for $k\ll n$, which proves the lemma.
	\end{proof}

	\subsubsection{Convergence of conditional expectation}
	We now show that $\Expn{a_{\varepsilon_n}(k,W_n^k(\delta))}$ as computed in Lemma~\ref{lem:expk} converges to a constant when we take the i.i.d.\ degrees into account.
	
	\begin{lemma}\label{lem:convexp}
		When $k\gg n^{(\tau-2)/(\tau-1)}$,
		\begin{equation}
			\frac{\Expn{a_{\varepsilon_n}(k,W_n^k(\delta))}}{n^{3-\tau}k^{\tau-3}}\plim c \mu^{2-\tau}\int_{\delta}^{1/\delta}x^{1-\tau}(1-\me^{-x})\dd x.
		\end{equation}
	\end{lemma}
\begin{proof}
	Define the random measure
	\begin{equation}
		\Mn[a,b]=\frac{1}{\mu^{1-\tau}n^{2-\tau}k^{\tau-1}}\sum_{u\in[n]}\ind{D_u\in[a,b]\mu n/k}.
	\end{equation}
	Since the degrees are i.i.d.\ samples from a power-law distribution, the number of vertices with degrees in interval $[a,b]$ is binomially distributed. Then,
	\begin{equation}
	\begin{aligned}[b]
		\Mn[a,b] & = \frac{1}{\mu^{1-\tau}n^{2-\tau}k^{\tau-1}}\abs{\left\{u: D_u\in[a,b]\mu n/k \right\}} \plim \frac{1}{(\mu n)^{1-\tau}k^{\tau-1}}\Prob{D\in[a,b]\mu n/k}\\
		& = \frac{1}{(\mu n)^{1-\tau}k^{\tau-1}}\int_{a\mu n/k}^{b\mu n/k}cx^{-\tau}\dd x = \int_{a}^b cy^{-\tau}\dd y =:\lambda[a,b],
		\end{aligned}
	\end{equation}
	where we used the change of variables $y=xk/(\mu n)$. By Lemma~\ref{lem:expk},
	\begin{equation}
	\begin{aligned}[b]
		\Expn{a_{\varepsilon_n}(k,W_n^k(\delta))}&=\frac{ \sum_{u\in W_n^k(\delta)}D_u(1-\me^{-D_uk/L_n})}{k} (1+\op(1))\\
		&=\frac{\mu n}{k}\frac{ \sum_{u\in W_n^k(\delta)}\frac{D_uk}{\mu n}(1-\me^{-D_uk/L_n})}{k} (1+\op(1))\\
		&=\frac{\mu^{2-\tau} n^{3-\tau}}{k^{3-\tau}}\int_{\delta}^{1/\delta}t(1-\me^{-t})\dd \Mn(t) (1+\op(1)).
	\end{aligned}
	\end{equation}
		Fix $\eta>0$. 
	Since $t(1-\me^{-t})$ is bounded and continuous on $[\delta,1/\delta]$, we can find $m<\infty $, disjoint intervals $(B_i)_{i\in[m]}$ and constants $(b_i)_{i\in[m]}$ such that $\cup B_i = [\delta,1/\delta]$ and
	\begin{equation}
	\Big|t(1-\me^{-t})-\sum_{i=1}^{m}b_i\ind{t\in B_i}\Big|<\eta/\lambda([\delta,1/\delta]),
	\end{equation}
	for all $t\in[\delta,1/\delta]$. 
	Because $\Mn(B_i)\plim \lambda(B_i)$ for all $i$, $\Mn(B_i)=\bigOp{\lambda(B_i)}$.
	Then,
	\begin{equation}
	\begin{aligned}[b]
	&\Big|\int_\delta^{1/\delta}t(1-\me^{-t})\dd \Mn(t)-\int_\delta^{1/\delta} t(1-\me^{-t})\dd \lambda(t)\Big|\\
	&\quad \leq \Big|\int_\delta^{1/\delta} t(1-\me^{-t})-\sum_{i=1}^mb_i \ind{t\in B_i}\dd \Mn(t)\Big|\\
	& \quad\quad + \Big|\int_\delta^{1/\delta} t(1-\me^{-t})-\sum_{i=1}^mb_i \ind{t\in B_i}\dd \lambda(t)\Big|\\
	&\quad \quad  +\Big|\sum_{i=1}^mb_i(\Mn(B_i)-\lambda(B_i))\Big|\\
	& \quad \leq  \eta\Mn([\delta,1/\delta])/\lambda([\delta,1/\delta]) +\eta +\op(\eta).
	\end{aligned}
	\end{equation}
	Using that $\Mn([\delta,1/\delta])=\bigOp{\lambda([\delta,1/\delta])}$ proves that
	\begin{equation}
		\int_{\delta}^{1/\delta}t(1-\me^{-t})\dd \Mn(t) \plim \int_{\delta}^{1/\delta}t(1-\me^{-t})\dd \lambda(t)=c\int_{\delta}^{1/\delta}x^{1-\tau}(1-\me^{-x})\dd x,
	\end{equation}
	which proves the lemma.
\end{proof}
		
	\subsubsection{Conditional variance of $a (k)$}\label{sec:var}
	We now show that the variance of $a_{\varepsilon_n}(k,W_n^k(\delta))$ is small when conditioning on the degree sequence, so that $a_{\varepsilon_n}(k,W_n^k(\delta))$ concentrates around its expected value computed in Lemma~\ref{lem:expk}.
	\begin{lemma}\label{lem:condvar}
		When $n^{(\tau-2)/(\tau-1)}\ll k \ll n^{1/(\tau-1)}$,
		\begin{equation}
			\frac{\Varn{a_{\varepsilon_n}(k,W_n^k(\delta))}}{\Expn{a_{\varepsilon_n}(k,W_n^k(\delta))}^2}\plim  0.
		\end{equation}
	\end{lemma}
\begin{proof}
	Again, it suffices to prove the lemma under the event $J_n$ and $A_n$ from~\eqref{eq:jn}. 
	We write the variance of $a_{\varepsilon_n}(k,W_n^k(\delta))$ as
	\begin{equation}\label{eq:varcond}
		\begin{aligned}[b]
			\Varn{a_{\varepsilon_n}(k,W_n^k(\delta))} & = \frac{1}{k^2|M_{\varepsilon_n}(k)|^2}\sum_{i,j\in M_{\varepsilon_n}(k)}\sum_{u,v\in W_n^k(\delta)}\Der_u\Der_w\\
			& \quad \times  (\Probn{X_{iu}=X_{jv}=1}-\Probn{X_{iu}=1}\Probn{X_{jv}=1})\\
			& = \frac{(1+\op(1))}{k^2|M_{\varepsilon_n}(k)|^2}\sum_{i,j\in M_{\varepsilon_n}(k)}\sum_{u,v\in W_n^k(\delta)}D_uD_w \\
			&\quad \times\left(\Probn{X_{iu}=X_{jv}=1}-\Probn{X_{iu}=1}\Probn{X_{jv}=1}\right).
			\end{aligned}
	\end{equation}
	Equation~\eqref{eq:varcond} splits into various cases, depending on the size of $\{i,j,u,v\}$. We denote the contribution of $\abs{\{i,j,u,v\}}=r$ to the variance by $V^{\sss{(r)}}(k)$. 
	We first consider $V^{\sss{(4)}}(k)$.
	We can write
	\begin{equation}
		\Probn{X_{iu}=X_{jv}=0}=\Probn{X_{iu}=0}\Probn{X_{jv}=0\mid X_{iu}=0}.
	\end{equation}
	For the second term, we first pair all half-edges adjacent to vertex $i$, conditionally on not pairing to vertex $u$. Then the second term can be interpreted as the probability that vertex $j$ does not pair to vertex $v$ in a configuration model with $L_n-D_i=L_n(1+o(1))$ vertices, where the degree of vertex $j$ is reduced by the amount of half-edges from vertex $i$ that paired to $j$. Similarly, the new degree of vertex $v$ is reduced by the amount of half-edges from vertex $i$ that paired to $v$. Since the expected number of half-edges from $i$ that pair to vertex $j$ is $O(D_iD_j/L_n)=D_jo(n^{-(\tau-1)/(\tau-1)})$~\cite{dorogovtsev2004}, the new degree of vertex $j$ is $D_j(1+\op(n^{-(\tau-2)/(\tau-1)}))$, and a similar statement holds for vertex $v$. Thus, by~\eqref{eq:pij} 
	\begin{equation}
	\Probn{X_{iu}=X_{jv}=0}=\me^{-D_iD_u/L_n}\me^{-D_jD_v/L_n}(1+\op(n^{-(\tau-2)/(\tau-1)})).
	\end{equation}
	This results in
	\begin{equation}
		\begin{aligned}[b]
		\Probn{X_{iu}=X_{jv}=1}&=1-\Probn{X_{iu}=0}-\Probn{X_{jv}=0}+\Probn{X_{iu}=X_{jv}=0}\\
		& = 1+(-\me^{-\frac{D_u k}{L_n}}-\me^{-\frac{D_v k}{L_n}}+\me^{-\frac{D_uk}{L_n}-\frac{D_vk}{L_n}})(1+\op(n^{-(\tau-2)/(\tau-1)}))\\
		& = (1-\me^{-D_uk/L_n})(1-\me^{-D_vk/L_n})(1+\op(1)),
		\end{aligned}
	\end{equation}
	where the last equality holds because $D_uk=\Theta(n)$ and $D_vk=\Theta(n)$ for $u,v\in W_n^k(\delta)$. Therefore
	\begin{equation}
		\begin{aligned}[b]
		V^{\sss{(4)}}(k)&  = \frac{1}{|M_{\varepsilon_n}(k)|^2k^2}\sum_{i,j\in M_{\varepsilon_n}(k)}\sum_{u,v\in W_n^k(\delta)}D_uD_v(1-\me^{-D_u k/L_n})(1-\me^{-D_v k/L_n})(1+\op(1)) \\
		&\quad - D_uD_v(1-\me^{-D_u k/L_n})(1-\me^{-D_u k/L_n})(1+\op(1)) \\
		& = \sum_{u,v\in W_n^k(\delta)}\op\left(k^{-2}D_uD_v(1-\me^{-D_u k/L_n})(1-\me^{-D_v k/L_n})\right)\\
		&	 = \op\left(\Expn{a_{\varepsilon_n}(k,W_n^k(\delta))}^2\right),
		\end{aligned}
	\end{equation}
	where the last equality follows from Lemma~\ref{lem:expk}.
	 Since there are no overlapping edges when $\{i,j,u,v\}=3$, $V^{\sss{(3)}}(k)$ can be bounded similarly. 
	
	We then consider the contribution from $V^{\sss{(2)}}$, which is the contribution where the two edges are the same. By Lemma~\ref{lem:convexp}, we have to show that this contribution is small compared to $n^{6-2\tau}k^{2\tau-6}$. We bound the summand in~\eqref{eq:varcond} as
	\begin{equation}\label{eq:ptriangij}
		\begin{aligned}[b]
			D_u^2 \left(\Probn{X_{iu}=1}-\Probn{X_{iu}=1}^2\right) &\leq D_u^2.          
		\end{aligned}
	\end{equation}
	Thus, using that on $A_n$, $|M_{\varepsilon_n}(k)|\geq 1$, $V^{\sss{(2)}}$, can be bounded as
	\begin{equation}\label{eq:V3}
		V^{\sss{(2)}}\leq \frac{1}{k^2|M_{\varepsilon_n}(k)|^2}\sum_{i\in M_{\varepsilon_n}(k)}\sum_{u\in W_n^k(\delta)}D_u^2=\frac{1}{k^2|M_{\varepsilon_n}(k)|}\sum_{u\in W_n^k(\delta)}D_u^2=\bigO{\frac{n^2}{k^4}}\abs{W_n^k(\delta)}.
	\end{equation}
		Since the degrees are i.i.d.\ samples from~\eqref{D-tail}, $\abs{W_n^k(\delta)}$ is distributed as a Binomial$(n,C(n/k)^{1-\tau})$ for some constant $C$. Therefore, $\abs{W_n^k(\delta)}=\bigOp{n\left(n/k\right)^{1-\tau}}$. This results in
	\begin{equation}
		V^{\sss{(2)}}=\bigOp{n^{4-\tau}k^{\tau-5}},
	\end{equation}
	which is smaller than $n^{6-2\tau}k^{2\tau-6}$ when $k\gg n^{\frac{\tau-2}{\tau-1}}$, as required.
\end{proof}

\begin{proof}[Proof of Proposition~\ref{prop:akmajor}]
	Lemma~\ref{lem:condvar} together with the Chebyshev inequality show that
	\begin{equation}
		\frac{a_{\varepsilon_n}(k,W_n^k(\delta))}{\Expn{a_{\varepsilon_n}(k,W_n^k(\delta))}}\plim 1.
	\end{equation}
	Combining this with Lemmas~\ref{lem:expk} and~\ref{lem:convexp} yields
	\begin{equation}
		\frac{a_{\varepsilon_n}(k,W_n^k(\delta))}{n^{3-\tau}k^{\tau-3}}\plim c\mu^{2-\tau}\int_{\delta}^{1/\delta}x^{1-\tau}(1-\me^{-x})\dd x.
	\end{equation}
\end{proof}

	\subsubsection{Contributions outside $W_n^k(\delta)$}\label{sec:minor}
	In this section, we prove Proposition~\ref{prop:minor} and show that the contribution to $a_{\varepsilon_n} (k)$ outside of the major contributing regimes as described in~\eqref{eq:wkn} is negligible. 
	\begin{proof}[Proof of Proposition~\ref{prop:minor}]
		We use that $\pprob_n({X}_{ij}=1)\leq \min(1,\frac{D_iD_l}{L_n})$. This yields
	\begin{equation}\label{eq:extriangbound}
	\begin{aligned}[b]
	\Exp{a_{\varepsilon_n}(k,\bar{W}_n^k(\delta))}& =\Exp{\Expn{a_{\varepsilon_n}(k,\bar{W}_n^k(\delta))}}\leq \frac{n}{k}\Exp{D\min\Big(1,\frac{kD}{L_n}\Big)\ind{D\in\bar{W}_n^k(\delta)}}\\
	& = \frac{n}{k}\int_{0}^{\delta \mu n/k}x^{1-\tau}\min\Big(1,\frac{k x}{\mu n}\Big) \dd x+\frac{n}{k}\int_{\mu n/(\delta k)}^{\infty}x^{1-\tau}\min\Big(1,\frac{k x}{\mu n}\Big) \dd x.
	\end{aligned}
	\end{equation}
	For ease of notation, we assume that $\mu=1$ in the rest of this section.	
	 We have to show that the contribution to~\eqref{eq:extriangbound} from vertices $u$ such that $D_u<\delta n/k$ or $D_u>n/(\delta k)$ is small. 
			First, we study the contribution to~\eqref{eq:extriangbound} for $D_u<\delta n/k$. 
		We can bound this contribution by taking the second term of the minimum, which bounds the contribution as
		\begin{equation}
		\int_{0}^{\delta n/k}x^{2-\tau}\dd x = \frac{\delta^{3-\tau}}{\tau-3}(k/n)^{\tau-3}.
		\end{equation}
		Then, we study the contribution for $D_u>n/(k\varepsilon)$. This contribution can be bounded very similarly by taking 1 for the minimum in~\eqref{eq:extriangbound}
		\begin{equation}
		\frac{n}{k}\int_{n/(\delta k)}^\infty x^{1-\tau}\dd x = \frac{\delta^{\tau-2}}{\tau-2}(k/n)^{\tau-3}.
		\end{equation}
		Taking $\kappa=\min(\tau-2,3-\tau)>0$ then proves the proposition. 
	\end{proof}

\subsection{Expected average nearest neighbor degree}\label{sec:akexp}
For $k\gg n^{(\tau-2)/(\tau-1)}$, it is easy to see that $\Exp{a (k,G_n)}$ satisfies~\eqref{eq:expak}. Indeed, Proposition~\ref{prop:minor} together with Lemma~\ref{lem:convexp} and taking the limit of $\delta\to 0$ establish~\eqref{eq:expak}. 

For $k$ small, note that the limit in Theorem~\ref{thm:ak}(i) has infinite mean, so that a similar scaling as in Theorem~\ref{thm:ak}(i) cannot be expected to hold for $\Exp{a (k,G_n)}$. 
To prove Theorem~\ref{thm:ak}(i), we used that the maximal degree scales as $n^{1/(\tau-1)}$ with high probability. When computing the expected average nearest neighbor degree however, the rare event of the maximal degree being larger than $n^{1/(\tau-1)}$ forms a major contribution to $\Exp{a(k,G_n)}$. In fact, we can follow the exact same lines as the proofs of Proposition~\ref{prop:minor} and Lemma~\ref{lem:convexp}, so that~\eqref{eq:expak} also holds for $k$ small.

\section{Proofs of Theorem~\ref{thm:akhvm} and~\ref{thm:akhrg}}\label{sec:akother}
We now briefly show how the proof of Theorem~\ref{thm:ak} can be adapted for the rank-1 inhomogeneous random graph and the hyperbolic random graph to prove Theorems~\ref{thm:akhvm} and~\ref{thm:akhrg}. We denote by $\pprob_n$ the probability conditioned on the weights in the rank-1 inhomogeneous random graph or conditioned on the radial coordinates in the hyperbolic model.

\subsection{Inhomogeneous random graph}
First, we show how to prove Theorem~\ref{thm:akhvm}(i). 
Similar to~\eqref{eq:DerD}, in the rank-1 inhomogeneous random graph the degree of a vertex with weight $h$, $D_h$, satisfies $D_h=h(1+\op(1))$ when $h\gg 1$~\cite{stegehuis2017}. Furthermore, the largest weight is of order $n^{1/(\tau-1)}$ with high probability. Thus, when $h\ll n^{(\tau-2)/(\tau-1)}$, w.h.p. $p(h,h')=h h'/(\mu n)$ for all vertices. When $u\in M_{\varepsilon_n}(k)$, $h_u=k(1+\op(1))$, so that conditionally on the weight sequence
\begin{equation}\label{eq:aksmallhvm}
\begin{aligned}[b]
a_{\varepsilon_n} (k)&=\frac{1}{k|M_{\varepsilon_n}(k)|}\sum_{u\in M_{\varepsilon_n}(k)}\sum_{i\in[n]}D_i\Probn{X_{iu}=1}\\
& =(1+\op(1))\frac{1}{k}\sum_{i\in [n]}h_i\frac{h_i k }{\mu n}=(1+\op(1))\sum_{i\in[n]}\frac{h_i^2}{\mu n},
\end{aligned}
\end{equation}
which is equivalent to~\eqref{eq:akDsize} because the weights are also sampled from~\eqref{D-tail}. This proves Theorem~\ref{thm:akhvm}(i). 

Similarly to~\eqref{eq:wkn}, we define for the rank-1 inhomogeneous random graph
\begin{equation}\label{eq:wnhvm}
W_n^{k,\sss\mathrm{HVM}}(\delta)=\{u:h_u\in [\delta\mu n/k, \mu n/(\delta k)] \}.
\end{equation}
Then it is easy to show that Proposition~\ref{prop:minor} also holds for the rank-1 inhomogeneous random graph with~\eqref{eq:wnhvm} instead of $W_n^k(\delta)$. We use that $\Probn{X_{ij}=1}=\min(h_ih_j/(\mu n),1)$. Because the weights are sampled from~\eqref{D-tail},~\eqref{eq:extriangbound} also holds for the rank-1 inhomogeneous random graph, so that Proposition~\ref{prop:minor} indeed holds for the rank-1 inhomogeneous random graph. 

We now sketch how to adjust the proof of Proposition~\ref{prop:akmajor} to prove an analogous version for the rank-1 inhomogeneous random graph, which states that
\begin{equation}\label{eq:akmajorhvm}
\frac{a_{\varepsilon_n}(k,W_n^{k\sss\mathrm{HVM}}(\delta))}{(n/k)^{3-\tau}}\plim c \mu^{2-\tau}\int_{\delta}^{1/\delta}x^{1-\tau}\min(x,1)\dd x.
\end{equation}
Following the proofs of Lemmas~\ref{lem:expk}-\ref{lem:condvar}, we see that these lemmas also hold for the rank-1 inhomogeneous random graph if we replace the connection probability of the erased configuration model of $1-\me^{-D_iD_j/L_n}$ by $\min(h_ih_j/\mu n,1)$. Note that for the rank-1 inhomogeneous random graph the contribution to~\eqref{eq:varcond} from 3 or 4 different vertices is 0, because the edge probabilities in the rank-1 inhomogeneous random graph conditioned on the weights are independent. From these lemmas,~\eqref{eq:akmajorhvm} follows. This then shows similarly to~\eqref{eq:aklimit} that
\begin{equation}
\frac{a_{\varepsilon_n}(k)}{(n/k)^{3-\tau}}\plim c \mu^{2-\tau}\int_{0}^{\infty}x^{1-\tau}\min(x,1)\dd x =\frac{ c \mu^{2-\tau}}{(3-\tau)(\tau-2)}.
\end{equation}
which proves Theorem~\ref{thm:akhvm}(ii). 

\subsection{Hyperbolic random graph}
We first provide a lemma that gives the connection probabilities conditioned on the radial coordinates in the hyperbolic random graph.
\begin{lemma}\label{lem:conprobhyp}
	For a hyperbolic random graph, the probability that $u$ and $v$ are connected conditionally on the radial coordinates can be written as
	\begin{equation}\label{eq:conprobhyp}
	\Probn{X_{uv}=1}=\min\left(\frac{1}{\pi} \cos^{-1}(1-2(\nu t(u)t(v)/n)^2),1\right)(1+\op(1)).
	\end{equation}
\end{lemma}
\begin{proof}
	Suppose $\nu t(u)t(v)/ n\geq 1$. Then, 
	\begin{equation}
	\frac{\nu t(u)t(v)}{n}=\frac{\nu \me^{R}\me^{-(r_u+r_v)/2}}{n}=\frac{n}{\nu}\me^{-(r_u+r_v)/2},
	\end{equation}
	so that $r_u+r_v\leq 2\log(N/\nu)=R$. Thus, by~\eqref{eq:dhyp}
	\begin{equation}
	\begin{aligned}[b]
	\cosh(d(u,v))& = \cosh(r_u)\cosh(r_v)-\sinh(r_u)\sinh(r_v)\cos(\theta_{uv})\\
	&\leq \cosh(r_u+r_v) \leq \cosh(R),
	\end{aligned}
	\end{equation}
	so that the distance between $u$ and $v$ is less than $R$ and $u$ and $v$ are connected. 
	
	Now suppose that $\nu t(u)t(v)/n<1$, so that $r_u+r_v>R$. 
	We calculate the maximal value of $\theta_{uv}$ such that $u$ and $v$ are connected, which we denote by $\theta_{uv}^*$. When the angle between $u$ and $v$ equals $\theta_{uv}^*$, the hyperbolic distance between $u$ and $v$ is precisely $R$. 
	Thus, we obtain, using the definition of the hyperbolic sine and cosine
	\begin{equation}\label{eq:distrurv}
	\frac{\me^{R}-\me^{-R}}{2}= \frac{\me^{r_u}-\me^{-r_u}}{2}\frac{\me^{r_v}-\me^{-r_v}}{2}-\frac{\me^{r_u}+\me^{-r_u}}{2}\frac{\me^{r_v}+\me^{-r_v}}{2}\cos(\theta_{uv}^*).
	\end{equation}
	Because $t(u)$ is distributed as~\eqref{eq:tudistr}, the maximal type scales is $\bigOp{n^{1/(\tau-1)}}$. Therefore, $\me^{r_u-r_v}=(t(v)/t(u))^2=\bigOp{n^{2/(\tau-1)}}$. Also, $\me^{-R}=O(n^{-2})$ so that~\eqref{eq:distrurv} becomes
	\begin{equation}\label{eq:distrurv2}
	\tfrac{1}{2}\me^{R} = \tfrac{1}{4}\me^{r_u+r_v}(1-\cos(\theta_{uv}^*))+\bigOp{n^{2/(\tau-1)}}.
	\end{equation}
	We then use that by the definitions of $t(u),t(v)$ and $R$
	\begin{equation}
		\me^{r_u+r_v}=\me^{R}\left(\me^{(r_u+r_v-R)/2}\right)^2=\me^{R}\Big(\frac{n}{\nu t(u)t(v)}\Big)^2.
	\end{equation}
	This yields for~\eqref{eq:distrurv2} that
	\begin{equation}
	1-\cos(\theta_{uv}^*)=2\Big(\frac{\nu t(u)t(v)}{n}\Big)^2+\bigOp{n^{-2\frac{\tau-2}{\tau-1}}\frac{\nu t(u)t(v)}{n}},
	\end{equation}
	so that
	\begin{equation}
	\theta_{uv}^* = \cos^{-1}(1-2(\nu t(u)t(v)/n)^2)(1+\op(1)).
	\end{equation}
	Because $u$ and $v$ are connected if their angular coordinates are at most $\theta_{uv}^*$ and the angular coordinates of $u$ and $v$ are sampled uniformly, we obtain that
	\begin{equation}
	\Probn{X_{uv}=1}= \frac{1}{\pi} \cos^{-1}(1-2(\nu t(u)t(v)/n)^2)(1+\op(1)).
	\end{equation}
\end{proof}

Using this lemma, we now prove Theorem~\ref{thm:akhrg}.

\begin{proof}[Proof of Theorem~\ref{thm:akhvm}]
	We first focus on $k\ll n^{(\tau-2)/(\tau-1)}$. By~\cite[Eq. (21)]{krioukov2010}, $\Expn{D_u}=\tfrac{\nu(\tau-1)}{\pi(\tau-2)}t(u)$. Thus, by~\cite[Theorem 2.7 and Lemma 3.5]{bringmann2015} $D_u=\tfrac{\nu(\tau-1)}{\pi(\tau-2)}t(u)(1+\op(1))$ when $t(u)\gg 1$. Therefore, when $u\in M_{\varepsilon_n}(k)$, $t(u)=\tfrac{\pi(\tau-2)}{\nu(\tau-2)}k(1+\op(1))$ when $k\gg 1$. Since the types are distributed as~\eqref{eq:tudistr}, the largest type is $\bigOp{n^{1/(\tau-1)}}$. Therefore, if $u\in M_{\varepsilon_n}(k)$, then $t(u)t(v)/n=\op(1)$ for all $v$. Taylor expanding~\eqref{eq:conprobhyp} shows that for $t(u)t(v)\ll n$
\begin{equation}
\Probn{X_{uv}=1}= \frac{2\nu t(u)t(v)}{\pi n}(1+\op(1)).
\end{equation}
Thus, similarly as in~\eqref{eq:aksmallhvm} we obtain with $\zeta=\pi/(2\nu)$,
\begin{equation}
\begin{aligned}[b]
a_{\varepsilon_n}(k)&=(1+\op(1))\frac{\nu(\tau-1)}{\pi(\tau-2) k}\sum_{i\in [n]}t(i)\frac{t(i)t(u) }{\zeta n} = (1+\op(1))\sum_{i\in[n]}\frac{t(i)^2}{\zeta n}.
\end{aligned}
\end{equation}
Combining this with the distribution of the types~\eqref{eq:tudistr} proves Theorem~\ref{thm:akhrg}(i) (which is the same as Theorem~\ref{thm:akhvm}(i) where $\mu$ is replaced by $\zeta$ and $c/(\tau-1)$ by 1).

We now investigate the case $k\gg n^{(\tau-2)/(\tau-1)}$. Similarly to~\eqref{eq:wkn}, we define for the hyperbolic random graph
\begin{equation}\label{eq:wnhrg}
W_n^{k,\sss\mathrm{HRG}}(\delta)=\{u: t(u)\in  [\delta \zeta n/k,\zeta  n/(\delta k)] \}
\end{equation}
with $\zeta=\pi/(2\nu)$. 
Using that $\cos^{-1}(1-2x^2)/\pi\leq x$ combined with Lemma~\ref{lem:conprobhyp}, we obtain
\begin{equation}
\Probn{X_{uv}=1}\leq \min\left(t(u)t(v)/(\zeta n),1\right).
\end{equation} 
Combining this with the fact that the $t(u)$s are sampled from a distribution similar to~\eqref{D-tail} shows that~\eqref{eq:extriangbound} also holds for the hyperbolic random graph, apart from a multiplicative constant. From there we can follow the same lines as the proof of Proposition~\ref{prop:minor}.

We can also prove an analogous proposition to Proposition~\ref{prop:akmajor} which states that
\begin{equation}\label{eq:akmajorhrg}
\frac{a_{\varepsilon_n}(k,W_n^{k,\sss\mathrm{HRG}}(\delta))}{(n/k)^{3-\tau}}\plim \frac{\nu (\tau-1)^2}{(\tau-2)\pi}\left(\frac{\pi}{2\nu}\right)^{2-\tau}\int_{\delta}^{1/\delta}x^{1-\tau}\min\left(\frac{1}{\pi}\cos^{-1}(1-2x^2),1\right)\dd x .
\end{equation}
 Because the connection probabilities conditioned on the radial coordinates in the hyperbolic model are given by Lemma~\ref{lem:conprobhyp}, a variant of Lemma~\ref{lem:expk} holds which states that
 \begin{equation*}
 \Expn{a_{\varepsilon_n}(k,W_n^k(\delta))}=\frac{\nu (\tau-1)}{k \pi(\tau-2)}\sum_{u\in W_n^k(\delta)}t(u)\min\left(\frac{1}{\pi} \cos^{-1}(1-2(\nu t(u)t(v)/n)^2),1\right)(1+\op(1))
 \end{equation*}
 Similarly, a variant of Lemma~\ref{lem:convexp} holds for the hyperbolic random graph, replacing the connection probability $1-\me^{-D_iD_j/(\mu n)}$ of the erased configuration model by~\eqref{eq:conprobhyp} and replacing the constant $c$ from~\eqref{D-tail} by its equivalent constant for the hyperbolic model of $\tau-1$ and $\mu$ by $\zeta$ (see~\eqref{eq:tudistr}). Furthermore, because conditionally on the radial coordinates, the probabilities that two distinct edges are present are independent, Lemma~\ref{lem:condvar} also holds for the hyperbolic random graph. Therefore, similar steps that lead to~\eqref{eq:aklimit} then show that
 \begin{equation}
	 \frac{a_{\varepsilon_n}(k)}{(n/k)^{3-\tau}}\plim \frac{\nu (\tau-1)^2}{(\tau-2)\pi}\left(\frac{\pi}{2\nu}\right)^{2-\tau}\int_{0}^{\infty}x^{1-\tau}\min\left(\frac{1}{\pi}\cos^{-1}(1-2x^2),1\right)\dd x 
 \end{equation} 
 which proves Theorem~\ref{thm:akhrg}(ii).
\end{proof}


\paragraph{Acknowledgements.}
This work was supported by NWO TOP grant 613.001.451. The author would like to thank Remco van der Hofstad and Johan S. H. van Leeuwaarden for their useful comments.

	\bibliographystyle{apt}
	\bibliography{../references}
	
\end{document}